\def\version{June 4, 2010}

\documentclass[reqno]{amsart}

\usepackage{amssymb,epsf,graphicx}

\newtheorem{theorem}{Theorem}[section]
\newtheorem{proposition}[theorem]{Proposition}
\newtheorem{lemma}[theorem]{Lemma}
\newtheorem{corollary}[theorem]{Corollary}

\theoremstyle{definition}

\theoremstyle{remark}

\newtheorem{remark}[theorem]{Remark}

\newcommand{\dd}{\,{\rm d}}
\newcommand{\e}[1]{\,{\rm e}^{#1}\,}

\def\Re{{\operatorname{Re\,}}}

\newcommand\eps\varepsilon

\newcommand{\C}{\mathbb{C}}

\renewcommand{\epsilon}{\varepsilon}

\newcommand{\R}{\mathbb{R}}

\renewcommand{\phi}{\varphi}

\DeclareMathOperator{\sgn}{sgn}

\begin{document}

{\hfill\small \version} \vspace{2mm}

\title[Schr\"odinger operators on the halfline]{A sharp bound on eigenvalues of Schr\"odinger operators on the halfline with complex-valued potentials}

\author[R. Frank]{Rupert L. Frank}
\address{Rupert L. Frank, Department of Mathematics, Princeton University, Princeton, NJ 08544, USA}
\email{rlfrank@math.princeton.edu}

\author[A. Laptev]{Ari Laptev}
\address{Ari Laptev, Department of Mathematics, Imperial College London, London SW7 2AZ, UK 
$\&$ Department of Mathematics, Royal Institute of Technology, 100 44 Stockholm, Sweden}
\email{a.laptev@imperial.ac.uk $\&$ laptev@math.kth.se}

\author[R. Seiringer]{Robert Seiringer}
\address{Robert Seiringer, Department of Physics, Princeton University, Princeton, NJ 08544, USA}
\email{rseiring@princeton.edu}

\thanks{\copyright\, 2009 by the authors. This paper may be reproduced, in its entirety, for non-commercial purposes.\\
Support through DFG grant FR 2664/1-1 (R.F.) and U.S. National Science Foundation grants PHY 0652854 (R.F.) and PHY 0652356 (R.S.) is gratefully acknowledged.}

\begin{abstract}
We derive a sharp bound on the location of non-positive eigenvalues of Schr\"odinger operators on the halfline with complex-valued potentials.
\end{abstract}

\maketitle

\section{Introduction and main result}

In this note we are concerned with estimates for non-positive eigenvalues of one-dimensional Schr\"odinger operators with complex-valued potentials. We shall provide an example of a bound where the sharp constant \emph{worsens} when a Dirichlet boundary condition is imposed. This is in contrast to the case of real-valued potentials, where the variational principle implies that the absolute value of the non-positive eigenvalues decreases.

In order to describe our result, we first assume that $V$ is real-valued. It is a well-known fact (attributed to L. Spruch in \cite{K}) that any negative eigenvalue $\lambda$ of the Schr\"odinger operator $-\partial^2 - V$ in $L^2(\R)$ satisfies
\begin{equation}
 \label{eq:spruch}
|\lambda|^{1/2} \leq \frac 12 \int_{-\infty}^\infty |V(x)| \dd x \,.
\end{equation}
The constant $\frac 12$ in this inequality is sharp and attained if $V(x)=c\delta(x-b)$ for any $c>0$ and $b\in\R$. (It follows from the Sobolev embedding theorem that the operator $-\partial^2 - V$ can be defined in the quadratic form sense as long as $V$ is a finite Borel measure on $\R$. In this case the right side of \eqref{eq:spruch} denotes the total variation of the measure.) From \eqref{eq:spruch} and the variational principle for self-adjoint operators we immediately infer that any negative eigenvalue of the operator $-\partial^2 - V$ in $L^2(0,\infty)$ with Dirichlet boundary conditions satisfies
\begin{equation}
 \label{eq:spruchd}
|\lambda|^{1/2} \leq \frac 12 \int_{0}^\infty |V(x)| \dd x \,.
\end{equation}
The constant $\frac 12$ in this inequality is still sharp but no longer attained.

Motivated by concrete physical examples and problems in computational mathematics, an increasing interest in eigenvalue estimates for \emph{complex-valued} potentials has developed in recent years. A beautiful observation of \cite{AAD} is that \eqref{eq:spruch} remains valid for all eigenvalues in $\C\setminus [0,\infty)$ even if $V$ is complex-valued. The same is not true for \eqref{eq:spruchd} ! Indeed, our main result is

\begin{theorem}
 \label{main}
For $a\in\R$ let
\begin{equation}
 \label{eq:g}
g(a) := \sup_{y\geq 0} \left| \e{iay} - \e{-y}\right|\,.
\end{equation}
Any eigenvalue $\lambda = |\lambda| e^{i\theta} \in\C\setminus [0,\infty)$ of the operator $-\partial^2 - V$ in $L^2(0,\infty)$ with Dirichlet boundary conditions satisfies
\begin{equation}
 \label{eq:main}
|\lambda|^{1/2} \leq \frac 12\, g(\cot(\theta/2))\, \int_0^\infty |V(x)| \dd x \,.
\end{equation}
This bound is sharp in the following sense: For any given $m>0$ and $\theta\in (0,2\pi)$ there are $c\in\C$ and $b>0$ such that for $V(x) = c \delta(x-b)$ one has $|c| = \int |V(x)| \dd x= m$ and the unique eigenvalue of $-\partial^2 - V$ is given by $ (m^2/ 4)\, g(\cot(\theta/2))^2 e^{i\theta}$, that is, equality is attained in \eqref{eq:main}.
\end{theorem}

\begin{remark}
  Our bound does not apply to positive eigenvalues. In the case of
  real-valued potential it is known that there are no positive
  eigenvalues if $V\in L^1(\R)$. 
\end{remark}

We note that $1<g(a)<2$ for $a>0$. The following lemma discusses the function $g$ in more detail.

\begin{lemma}\label{g}
For $a\geq 0$, the function $g(a)$ is monotone increasing, with $g(0)= 1$ and $\lim_{a\to \infty} g(a) = 2$. Moreover,
\begin{equation}\label{eq:g0}
g(a) = 1 + O(\e{-\pi/(3a)})
\end{equation}
for small $a$, and 
\begin{equation}\label{eq:ginfty}
g(a) = 2 - \frac \pi a + O(a^{-2}) 
\end{equation}
as $a\to \infty$.
\end{lemma}

In Figure~\ref{Fig.1} we plot the curve $\{|z| = g(\cot(\theta/2))^2\}$. It follows from \eqref{eq:ginfty} that this curve hits the positive real axis at the point $4$ with slope $2/\pi$. Close to the point $-1$ the curve coincides with a semi-circle up to exponentially small terms, as \eqref{eq:g0} shows. 

\begin{figure}[htf]
\includegraphics[width=9cm, height=4.5cm]{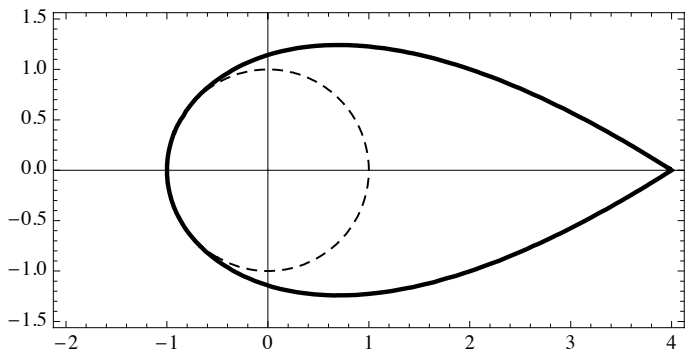}
\caption{The maximal value of $4|\lambda|$ on the half-line with $\int_0^\infty |V(x)|\dd x=1$. The dashed line is the corresponding bound on the whole line.}
\label{Fig.1}
\end{figure}

Using that $\sup_a g(a)=2$ we find

\begin{corollary}
 Any eigenvalue $\lambda \in\C\setminus [0,\infty)$ of the operator $-\partial^2 - V$ in $L^2(0,\infty)$ with Dirichlet boundary conditions satisfies
\begin{equation}
 \label{eq:maincor}
|\lambda|^{1/2} \leq \int_0^\infty |V(x)| \dd x \,.
\end{equation}
The bound is not true in general if the right side is multiplied by a constant $<1$.
\end{corollary}

Inequality \eqref{eq:maincor} follows also from inequality \eqref{eq:spruch} for complex-valued potentials. Indeed, the odd extension of an eigenfunction of the Dirichlet operator is an eigenfunction of the whole-line operator with the potential $V(|x|)$ with the same eigenvalue. The remarkable fact is that the inequality is sharp in the complex-valued case, as shown in Theorem \ref{main}.

By the same argument \eqref{eq:maincor} is also valid if \emph{Neumann} instead of Dirichlet boundary conditions are imposed. In this case equality holds for any $V(x) = c \delta(x)$ with $\Re c>0$. In particular, in the Neumann case (\ref{eq:maincor}) is sharp for any fixed argument $0<\theta<2\pi$ of the eigenvalue $\lambda$. The analogue for mixed boundary conditions is

\begin{proposition}\label{prop1}
Let $\sigma \geq 0$. Any eigenvalue $\lambda \in\C\setminus [0,\infty)$ of the operator $-\partial^2 - V$ in $L^2(0,\infty)$ with boundary conditions $\psi'(0) = \sigma \psi(0)$ satisfies
\begin{equation}\label{eq:maincor2}
|\lambda|^{1/2} \leq \int_0^\infty |V(x)| \dd x \,.
\end{equation}
The bound is sharp for any $\sigma\geq 0$ and any fixed argument $0<\theta<2\pi$ of the eigenvalue $\lambda$.
\end{proposition}

Note that if $ \sigma <0$ a bound of the form (\ref{eq:maincor2}) can not hold since there exists a non-positive eigenvalue even in the case $V = 0$.

\begin{remark}
 In the \emph{self-adjoint} case inequality \eqref{eq:spruch} for whole-line operators is accompanied by bounds
\begin{equation}
 \label{eq:keller}
|\lambda|^\gamma \leq \frac{\Gamma(\gamma+1)}{\sqrt\pi\, \Gamma(\gamma+3/2)} \left(\frac{\gamma-1/2}{\gamma+1/2}\right)^{\gamma-1/2} \int_{-\infty}^\infty |V(x)|^{\gamma+1/2} \dd x \
\end{equation}
for $\gamma>1/2$; see \cite{K, LT}. In contrast, in the \emph{non-selfadjoint} case it seems to be unknown whether the condition $V\in L^{\gamma+1/2}(\R)$ for some $1/2<\gamma<\infty$ implies that all eigenvalues in $\C\setminus[0,\infty)$ lie inside a finite disc; see \cite{DN,FLLS,LS,S} for partial results in this direction. We would like to remark here that even if a bound of the form \eqref{eq:keller} were true in the non-selfadjoint case with $1/2<\gamma<\infty$, then (in contrast to \eqref{eq:spruch} for $\gamma=1/2$) the constant would have to be strictly larger than in the self-adjoint case. To see this, consider $V(x)=\frac{\alpha(\alpha+1)}{\cosh^2 x}$ with $\Re\alpha>0$. Then $\lambda=-\alpha^2$ is an eigenvalue (with eigenfunction $(\cosh x)^{-\alpha}$) and the supremum
$$
\sup_{\Re\alpha\geq 0} \frac{|\lambda|^\gamma}{\int_{-\infty}^\infty |V(x)|^{\gamma+1/2} \dd x}
= \left( \int_{-\infty}^\infty \frac{dx}{\cosh^2 x} \right)^{-1} \sup_{\Re\alpha \geq 0} \frac{|\alpha|^{\gamma-1/2}}{|\alpha+1|^{\gamma+1/2}}
$$
is clearly attained for purely imaginary values of $\alpha$.
\end{remark}


\section{Proofs}

\begin{proof}[Proof of Theorem \ref{main}]
Assume that $-\partial^2 \psi(x) - V(x)\psi(x) = - \mu \psi(x)$ with $\psi(0) = 0$, $\psi\not\equiv 0$ and $\mu =-\lambda\in \C \setminus (-\infty,0]$. Then the Birman-Schwinger operator
$$
V^{1/2} \frac 1{-\partial^2 + \mu} |V|^{1/2} \,,
\qquad V^{1/2} := (\sgn V) |V|^{1/2} \,,
$$
has an eigenvalue $1$, and hence its operator norm is greater or equal to 1.

The integral kernel of this operator equals 
$$
V(x)^{1/2} \frac{ \e{-\sqrt\mu|x-y|} - \e{-\sqrt\mu (x+y) }}{2\sqrt\mu} |V(y)|^{1/2}\,,
$$
and hence
$$
\left|\left( \psi\,, V^{1/2} \frac 1{-\partial^2 + \mu} |V|^{1/2}\, \phi \right) \right| \leq \frac {\|V\|_1} {2\sqrt{|\mu|}} \|\psi\|_2 \|\phi\|_2 \sup_{x, y\geq 0} \left| \e{-\sqrt\mu|x-y|} - \e{-\sqrt\mu (x+y) } \right|\,.
$$
Without loss of generality, we can take the supremum over the smaller set $x\geq y\geq 0$. Then
$$
\sup_{x\geq y\geq 0} \left| \e{-\sqrt\mu(x-y)} - \e{-\sqrt\mu (x+y) } \right| = \sup_{x\geq y\geq 0} \e{-x\,\Re \sqrt\mu} \left| \e{\sqrt\mu y} - \e{-\sqrt\mu y } \right|\,.
$$
Since $\Re\sqrt\mu>0$, the supremum over $x$ is achieved at $x=y$, and hence
$$
\sup_{x, y\geq 0} \left| \e{-\sqrt\mu(x-y)} - \e{-\sqrt\mu (x+y) } \right| = \sup_{y\geq 0} \left| 1 - \e{-2\sqrt\mu y}\right|\,.
$$
If we write $\mu = -|\mu|\e{i\theta}$ with $0<\theta<2\pi$, then
$$
 \sup_{y\geq 0} \left| 1 - \e{-2\sqrt\mu y}\right| = \sup_{y\geq 0} \left| \e{2 i \sqrt{|\mu|} \cos(\theta/2) y} - \e{-2 \sqrt{|\mu|} \sin(\theta/2) y}\right| = g(\cot(\theta/2))
$$
with $g$ from \eqref{eq:g}. Hence we have shown that
\begin{equation}\label{ineq}
\left\| V^{1/2} \frac 1{-\partial^2 + \mu} |V|^{1/2}\, \right\| \leq \frac {\|V\|_1} {2\sqrt{|\mu|}}\ g(\cot(\theta/2))
\,.
\end{equation}
Since the left side is greater or equal to 1, as remarked above, we obtain \eqref{eq:main}.

For $V(x)=c\delta(x-b)$ the Birman-Schwinger operator reduces to the number $c(1-e^{-2\sqrt\mu b})/(2\sqrt{\mu})$ and inequality (\ref{ineq}) becomes equality provided $\sqrt\mu b$ satisfies $|1-e^{-2\sqrt\mu b}| = g(\cot(\theta/2))$. For given $m>0$ and $\theta\in(0,2\pi)$ this determines $b$ and $|c|$. The phase of $c$ is found from the equation
$c(1-e^{-2\sqrt\mu b})/(2\sqrt{\mu})=1$.
\end{proof}

\begin{proof}[Proof of Lemma \ref{g}]
By continuity for $a>0$ there exists an optimizer $y_0$ such that $g(a) = |\e{iay_0} - \e{-y_0}|$. We claim that $y_0$ satisfies $\pi/3 < ay_0 \leq \pi$. To see the lower bound, note that $|\e{iay}-\e{-y}|\geq 1$ if and only if $2\cos(ay)\leq \e{-y}$. In particular, $\cos(ay_0)<1/2$.
For the upper bound,  if $2\pi > ay>\pi$ and $2\cos(ay)<\e{-y}$, replacing $ya$ by $2\pi - ya$ leads to a  contradiction. Similarly, if $ya>2\pi$ it can replaced by $ya-2\pi$ in order exclude that $y$ is the optimizer. 

It is elementary to check that $|\e{iay}-\e{-y}|$ is monotone increasing in $a$ for every fixed $y$ with $0\leq y\leq \pi/a$. Since we know already that $y_0\leq \pi/a$,  the monotonicity of $g$ follows.

Plugging in $y=\pi/a$, we obtain $g(a)\geq  1+\e{-\pi/a}\geq 2-\pi/a$. 
For large enough $a$, it follows from this that $y_0$ is close to $\pi/a$. In particular, $y_0\geq \pi/(2a)$,  and hence  $|\e{iay_0}-1|\geq g(a) \geq 2-\pi/a$. This implies that $y_0 = \pi/a + O(a^{-2})$, and thus $g(a) = 2 - \pi/a + O(a^{-2})$, as claimed.

For an upper bound for small $a$, we use the triangle inequality and the bound $ay_0\geq \pi/3$ to find $g(a) \leq 1 + e^{-y_0} \leq 1 + \e{-\pi/(3a)}$. 
\end{proof}

\begin{proof}[Proof of Proposition~\ref{prop1}]
We proceed as in the proof of Theorem~\ref{main}. The Birman-Schwinger operator has the kernel
$$
V(x)^{1/2} \frac{ \e{-\sqrt\mu|x-y|} + \frac{\sqrt\mu - \sigma}{\sqrt\mu + \sigma} \e{-\sqrt\mu (x+y) }}{2\sqrt\mu} |V(y)|^{1/2}\,.
$$
The assertion follows as above using that
$$
\sup_{y\geq 0} \left| 1 + \frac{\sqrt\mu - \sigma}{\sqrt\mu + \sigma} \e{- 2 \sqrt\mu y } \right| \leq 2
$$
by the triangle inequality and the fact that $|\sqrt\mu-\sigma|\leq |\sqrt\mu+\sigma|$. The fact that the bound (\ref{eq:maincor2}) is sharp for given argument $0<\theta<2\pi$ of the eigenvalue $\lambda$ follows by choosing $V(x) = - c i \e{i\theta/2} \delta(x)$ for $c>0$ and letting $c\to \infty$.  
\end{proof}

\end{document}